\documentclass[11pt]{amsart}

\usepackage[pdftex]{graphicx}
\usepackage{amssymb}
\usepackage{amsmath}
\usepackage{amsfonts}
\usepackage{amsthm}
\usepackage{amssymb}
\usepackage{mathrsfs}
\newtheorem{theorem}{Theorem}[section]
\newtheorem*{theorem*}{Theorem}
\newtheorem{proposition}[theorem]{Proposition}
\newtheorem{corollary}[theorem]{Corollary}

\theoremstyle{definition}
\newtheorem{definition}[theorem]{Definition}
\theoremstyle{remark}
\newtheorem{remark}[theorem]{Remark}
\newtheorem{example}[theorem]{Example}

\newcommand{\mA}{\mathcal{A}}

\newcommand{\mH}{\mathcal{H}}
\newcommand{\mF}{\mathcal{F}}
\newcommand{\Res}{\mathrm{Res}}

\newcommand{\bS}{\mathbb{S}}
\newcommand{\bbR}{\mathbb{R}}

\newcommand{\frakp}{\mathfrak{p}}

\newcommand{\CC}{\mathbb{C}}

\begin{document}
\title{Spherical conical metrics and harmonic maps to spheres}
\begin{abstract}
A spherical conical metric $g$ on a surface $\Sigma$ is a metric of constant curvature $1$ with finitely many isolated conical singularities. The uniformization problem for such metrics remains largely open when at least one of the cone angles exceeds $2\pi$. The eigenfunctions of the Friedrichs Laplacian $\Delta_g$ with eigenvalue $\lambda=2$ play a special role in this problem, as they represent local obstructions to deformations of the metric $g$ in the class of spherical conical metrics. In the present paper we apply the theory of multivalued harmonic maps to spheres to the question of existence of such eigenfunctions. In the first part we establish a new criterion for the existence of $2$-eigenfunctions, given in terms of a certain meromorphic data on $\Sigma$. As an application we give a description of all $2$-eigenfunctions for metrics on the sphere with at most three conical singularities. The second part is an algebraic construction of metrics with large number of $2$-eigenfunctions via the deformation of multivalued harmonic maps. We provide new explicit examples of metrics with many $2$-eigenfunctions via both approaches, and describe the general algorithm to find metrics with arbitrarily large number of $2$-eigenfunctions.

%
\end{abstract}

\author[M. Karpukhin]{Mikhail Karpukhin}
\address{Mathematics 253-37, Caltech, Pasadena, CA 91125, USA
}
\email{mikhailk@caltech.edu}
\author[X. Zhu]{Xuwen Zhu}
\address{Department of Mathematics, Northeastern University, Boston, MA 02115}
\email{x.zhu@northeastern.edu}
\maketitle

\section{Introduction}
The study of spectral behaviour of metrics with singularities has a long history and has interesting applications in geometry. In this paper, we consider positive constant curvature metrics on compact Riemann surfaces with conical singularities, which is a subject that has attracted considerable attention in recent years. One central question is the following ``singular uniformization problem'': given a set of conical data consisting of cone angles and cone positions on a Riemann surface, does there exist a positive constant curvature metric with prescribed conical singularities? The answer, when at least one of the cone angles is bigger than $2\pi$, is complicated and not completely understood. One difficulty is that  the linearized operator of the nonlinear Liouville equation for positive constant curvature metrics, given by $\Delta-2$,  is not always invertible (here $\Delta$ is the Laplace--Beltrami operator).  In the work of Mazzeo and the second author~\cite{MZ1, MZ2}, the deformation of such metrics and the local structure of moduli spaces was studied, and it was shown that the local deformations of a spherical conical metric are determined by the spectral behavior of the Laplacian. In particular, the metric can deform freely if $2$ is not in the spectrum of the Friedrichs Laplace operator, and has local deformation obstructions otherwise. 
Moreover, there is a pairing formula that uses the local expansions of these eigenfunctions with eigenvalue $2$ to dictate the exact directions of obstructions.  Therefore, in order to understand the singular locus of the moduli space, it is natural to ask if one can find special spherical conical metrics with eigenvalue $2$ and determine the multiplicity of the eigenvalue. 

Following~\cite{UY, Eremenko, CWWX} we study conical spherical metrics $g$ using their {\em developing map}. Roughly speaking,  a developing map of $g$ is a multivalued meromorphic map $f$ on $\Sigma$ such that $g=f^*g_{\mathrm{st}}$, where $g_{\mathrm{st}}$ is constant curvature $1$ metric on $\overline{\mathbb{C}}$, see Section~\ref{s:SphCo} for more details. It is easy to see that the components of the stereographic projection of $f$ are bounded  $2$-eigenfunctions of $\Delta_g$. In general, these components are all multivalued. However, when the monodromy of $f$ has a simple form, some of these components are single-valued. In particular, if the monodromy is trivial (or reducible), then the developing map $f$ gives rise to $3$-dimensional ($1$-dimensional) family of $2$-eigenfunctions.  The local obstructions induced by these ``trivial'' eigenfunctions are easy to compute. However, in order to determine the admissible directions of deformation one needs to know all the 2-eigenfunctions.  Thus, a natural question is whether there exist any other $2$-eigenfunctions. This brings us to the following definition.
%
\begin{definition}
\label{def:extra}
A function $u$ in the Friedrichs domain of a spherical conical metric $g$ is said to be an {\em extra eigenfunction} if:
\begin{itemize}	
\item $u$ is a 2-eigenfunction for the Beltrami-Laplace operator of $g$, i.e. $\Delta_{g}u=2u$; 
\item $u$ is not obtained by the stereographic projection of a developing map of $g$.
\end{itemize}
\end{definition}

If the monodromy is trivial, then the developing map $f$ is a branched cover, see e.g.~\cite{CWWX}. For branched covers the existence of extra eigenfunctions has been extensively studied using the theory of minimal surfaces and harmonic maps, see e.g.~\cite{Ejiri, Kotani, MR, Nayatani} and references therein. The main goal of the present paper is to extend these results to the case of non-trivial monodromy.

In the first part of this paper we show a criterion for the existence of extra eigenfunctions in terms of certain meromorphic data on $\Sigma$. This can be seen as a generalization of the results in~\cite{MR, Nayatani}. One first observes that any extra eigenfunction can be written as a support function of a minimal surface in $\mathbb{R}^3$ with Gauss map $f$. Thus, the existence of such minimal surfaces can then be studied using the Weierstrass representation. For the exact statement of the criterion we refer to Theorem~\ref{th1}. As an application we obtain the following.

\begin{theorem}
Let $g$ be a conical spherical metric on the sphere $\mathbb{S}^2$ with at most $3$ cone points. Then $g$ does not have any extra eigenfunctions.
\end{theorem}    
The criterion in Theorem~\ref{th1} is especially effective when most cone points of $g$ have cone angles that are multiple of $2\pi$. In particular, it allows us to construct explicit examples of spherical conical  metric with extra eigenfunctions when there are exactly two non-integer angles, see Example~\ref{ex1}.

In the second part of this paper we provide an algebraic  construction of spherical metrics with extra eigenfunctions by using the deformation theory of harmonic maps to spheres. This approach is based on viewing the developing map $f$ as a multivalued harmonic map to $\mathbb{S}^2$ considered as a subset of $\mathbb{S}^n$ and then constructing $f$ as a limit of linearly full multivalued harmonic maps to $\mathbb{S}^n$. We refer to Theorem~\ref{S2deformation:thm} for the precise statement. These ideas have been previously applied for branched covers of $\mathbb{S}^2$, see~\cite{Ejiri0,Ejiri,Kotani} and references therein. In the case of branched covers all harmonic maps are single-valued, thus, our results can be seen as a generalization to multivalued setting. 

Theorem~\ref{S2deformation:thm} provides an efficient way of constructing both the metric and the extra eigenfunctions. We give two applications. First, an explicit family of spherical conical  metrics and the expressions for their extra eigenfunctions can be found in Example~\ref{intro:example}. Second, the same construction yields the existence of spherical metrics with arbitrary many extra eigenfunctions and Section~\ref{s:algorithm} contains an algorithm designed to find such metrics.

Finally, we remark that the techniques to study harmonic maps developed in~\cite{Ejiri, Kotani, MR, Nayatani} have proved to be extremely useful in the study of sharp isoperimetric eigenvalue inequalities, where the $2$-eigenfunctions of $\Delta_g$ also play a special role. We refer to the papers~\cite{KarpukhinRP2, KNPP, NS} for applications to several long-standing conjectures. 


\subsection*{Organization of the paper} This paper is organized as follows. In~\S\ref{s:SphCo} we review the geometry of spherical conical metrics and existing results on the related spectral theory problems. In~\S\ref{s:criterion} we prove the equivalence condition in Theorem~\ref{th1} and discuss an explicit example. In~\S\ref{s:deform} we connect the existence of extra eigenfunctions with the deformation of harmonic maps to prove Theorem~\ref{CS2deformation:thm} and give another construction that produces an arbitrarily large number of eigenfunctions.  

\noindent\textbf{Acknowledgements: } We would like to thank Dima Jakobson for suggesting the collaboration. The second author is grateful to Rafe Mazzeo and Bin Xu for many helpful discussions. The first author is partially supported by NSF DMS-1363432. The second author is partially supported by NSF DMS-2041823.

\section{Spherical conical metrics}\label{s:SphCo}

In this section we give a brief review of the geometry of spherical conical metrics and its spectral behavior; for more details we refer to~\cite{BMM, Ere2, MZ2, MP2} and the reference therein. Let $\Sigma$ be a compact Riemann surface, $\vec \beta = (\beta_{1}, \dots, \beta_{k})\in \bbR_{+}^{k}$ a $k$-tuple and $\frakp=\{p_{1}, \dots, p_{k}\}$ a set of $k$ distinct points on $\Sigma$, a spherical conical metric $g$ with conical data $(\Sigma, \vec \beta, \frakp)$ is a constant curvature one metric on the punctured surface 
\begin{equation}\label{e:sigmap}
\Sigma_{\frakp}:=\Sigma\setminus \{p_{1}, \dots, p_{k}\}
\end{equation} 
with conical singularities at each $p_{j}$ with angle $2\pi\beta_{j}$. Here a conical singularity means that there exists a conformal coordinate centered at $p_{j}$ such that locally the metric is given by $g=e^{2u}|dz|^2$ where $u-(\beta_j-1)\log\,|z|$ extends to $z=0$ continuously. 

Spherical conical metrics can be identified with solutions to the singular Liouville equation:
\begin{equation}
\Delta_{g_{0}}u-e^{2u}+K_{g_{0}}=0
\end{equation}
where the $\Delta_{g_{0}}$ is the Laplace-Beltrami operator in the Friedrichs extension for a model metric $g_{0}$. Here the Friedrichs extension is the self-adjoint extension of $\Delta_{g_{0}}$ from compactly supported smooth functions on the punctured surface  $\Sigma_{\frakp}$ to bounded functions on $\Sigma$. The existence and uniqueness of solutions to the above equation does not always hold. To understand these solutions and corresponding metrics from the deformation point of view, in~\cite{MZ2} the second author and collaborator studied the linearization of this nonlinear equation, given by the following operator
$$
\Delta_{g}-2
$$
where $g$ is a spherical conical metric that one would like to perturb. 
Using techniques developed in~\cite{MZ1} it was shown that the local perturbation theory around a spherical metric $g$ is closely related to the kernel of the linear operator above. In particular, the dimension of the kernel  and the local expansion of these eigenfunctions determines the obstruction of the local deformation (cf. \cite[Theorem 4]{MZ2}).  This provides the motivation for us to identify all these eigenfunctions, as the spherical conical metrics with those eigenfunctions are where the singularities of the moduli space occur.   

One of the first understanding towards this direction was through the study of a special class of spherical conical metrics called reducible metrics. Such metrics are defined by the monodromy of their developing maps. Here we briefly recall the definition.
For any spherical conical metric $g$, there exists a multivalued meromorphic function
$$f: \Sigma_{\frakp}  \to \overline{\mathbb{C}}$$
called a {\it developing map} of $g$, defined by the following three conditions (cf. \cite[Lemma 2.1 and Lemma 3.1]{CWWX}):
\begin{enumerate}
\item (Pull-back) Denote the standard metric on the sphere by $g_{\rm st}=\frac{4|dw|^2}{(1+|w|^2)^2}$ for $w\in \overline{\CC}$, then  the metric $g$ is given by the pullback by $f$, i.e. $g=f^{*}g_{\rm st}$  on $\Sigma_{\frakp}$;
\item (Monodromy) The monodromy of $f$ is contained in ${\rm PSU}(2)$;
\item (Angle) Near each $p_j$, the principal singular term of the Schwarzian derivative of $f$ is given by $\frac{1-\beta_j^2}{2z^2}$.
\end{enumerate}
We note here that for a given spherical conical metric, its developing map is not unique, and all such maps are related by M\"obius transformations in ${\rm PSU}(2)$. For a given metric, the monodromy of all its developing maps are contained in the same conjugacy class of ${\rm PSU}(2)$.
\begin{definition}
A spherical conical metric $g$ is called {\it reducible} if there exists  a developing map with monodromy in  ${\rm U}(1)$. 
The metric $g$ is called {\it trivially reducible} if the monodromy of its developing map is trivial.
\end{definition}

Let us now clarify the discussion before Definition~\ref{def:extra} and show how to obtain $2$-eigenfunctions from a developing map. 
Recall the definition of harmonic maps which will be used later.
\begin{definition}
Let $\phi: (M,g)\rightarrow (N,h)$ be a smooth map between two manifolds. We say $\phi$ is a {\em harmonic map} if it is a critical point for the energy
$$
\frac{1}{2}\int_{M}|d\phi|_{g,h} dv_{g}.$$
\end{definition}
When the target $(N,h)$ is the unit sphere $\mathbb{S}^{n}$, a standard computation shows that $\phi$ is harmonic if and only if
$$
\Delta_{g}\phi = |d \phi|^{2}_{g}\phi.
$$
If in addition $M$ is a surface, one can let $g_{\phi}=\frac{1}{2}|d\phi|_{g}^{2}g$ then by the conformal invariance of Laplacian we have
$$
\Delta_{g_{\phi}} \phi= 2\phi,
$$
that is, $\phi$ is a 2-eigenfunction for metric $g_{\phi}$. Now we connect this to the developing maps of a spherical metric.

 Let $f$ be a developing map viewed as a multivalued meromorphic function on 
 $\Sigma_{\frakp}$. 
 Applying stereographic projection, we can identify $f$ with a multivalued harmonic map 
 $\phi\colon\Sigma_{\frakp} \rightarrow \bS^{2}\subset\mathbb{R}^3$ defined as 
\begin{equation}\label{e:phi}
\phi: = \frac{1}{|f|^2+1}(2\Re f, 2\Im f, |f|^2-1).
\end{equation}
The induced metric $g=g_\phi$ is given by
$$
g=\frac{4|f'|^{2}}{(|f|^{2}+1)^{2}}|dz|^{2}
$$ 
and the Laplace-Beltrami operator is given by
$$
\Delta_{g}u=-\frac{(|f|^{2}+1)^{2}}{|f'|^{2}}u_{z\bar z}.
$$
Write the harmonic map as $\phi=(\phi_{1}, \phi_{2}, \phi_{3})$, then the three (potentially multivalued) components $\phi_{j}$ satisfy $(\Delta-2)\phi_{j}=0$ pointwise for $j=1,2,3$.
 
We also remark that
$$
|\phi_{z}|^{2}=\frac{2|f'|^{2}}{(|f|^{2}+1)^{2}}.
$$
Hence the following relation is true and will be used later
\begin{equation}\label{e:eigen}
\Delta_{g}u=2u \ \Leftrightarrow \ u= -u_{z\bar z}/|\phi_{z}|^{2}.
\end{equation}

If the metric is reducible, then there is at least one point $s\in\mathbb{S}^2$ such that it is fixed by the monodromy group. As a result the inner product $(\phi,s)$ is invariant under the monodromy group and is a well-defined bounded function on $\Sigma_{\frakp}$, i. e. it is in the domain of the Friedrichs extension. This observation was made in~\cite{XZ}.  
 If the metric is trivially reducible, then $(\phi,s)$ is well defined for any $s\in \bS^{2}$, hence the components give three-dimensional space of eigenfunctions.

\begin{remark}
Even through a priori none of the three components of $\phi$ need to be single-valued for a conical metric $g$, the following is true: if there exists $s\in \mathbb{S}^{2}$ such that $(\phi,s)$ is a single-valued function on $g$, then $g$ is reducible and $s$ is a fixed point of the monodromy.
\end{remark}

\begin{remark}\label{r:mono}
Under the identification~\eqref{e:phi}, there will be two kinds of monodromy groups used in this paper: $\operatorname{PSU}(2)$ acting on $\overline{\mathbb{C}}$ for the developing map $f$, and $\operatorname{SO}(3)$ acting on $\mathbb{S}^2 \subset\mathbb{R}^{3}$ for the harmonic map $\phi$. The action of $\operatorname{PSU}(2)$ and $\operatorname{SO}(3)$ are related by the classical identification, i.e. $\mathfrak{su}(2)\equiv \mathbb{R}^{3}$ through the basis of Pauli matrices and $\operatorname{PSU}(2)$ acts on $x\in \mathfrak{su}(2)$ via conjugation which is indeed an orthogonal transformation. Specifically, for any $A\in \operatorname{PSU(2)}$ written as
$$
A=
\begin{bmatrix}
\alpha & \beta\\
-\bar\beta & \bar \alpha 
\end{bmatrix}, \
|\alpha|^{2}+ |\beta|^{2}=1,
$$
it acts on $f: \Sigma_{\frakp} \to \overline{\mathbb{C}}$ as
$$
Af=\frac{\alpha f +\beta}{-\bar \beta f+ \bar \alpha}.
$$
The corresponding monodromy of a harmonic map $\phi$ is given by
$$
\mathcal{A}\phi=\frac{1}{|f|^{2}+1}\mathcal{A}\begin{pmatrix}
2\Re f\\
2 \Im f\\
|f^{2}|-1
\end{pmatrix}
$$
where the matrix $\mathcal{A}$ is
\begin{equation}\label{e:monoid}
\mathcal{A}=\begin{bmatrix}
\Re (\alpha^{2}-\beta^{2}) & \Im (\alpha^{2}+\beta^{2})& -2\Re (\alpha\beta) \\
-\Im(\alpha^{2}-\beta^{2}) & \Re (\alpha^{2}+\beta^{2}) & 2\Im(\alpha\beta)\\
2\Re(\alpha\bar \beta) & 2\Im (\alpha \bar \beta) & |\alpha|^{2}-|\beta^{2}|
\end{bmatrix}
\in\operatorname{SO}(3)
\end{equation}
which is exactly the classical identification with $A\in \operatorname{PSU}(2)$.
For more explanations related to spherical conical metrics, see for example~\cite[Section 2]{MP}. In this paper we use these two monodromy groups interchangeably. 
\end{remark}

In the rest of the paper we are interested in the following question: are there any other eigenfunctions with eigenvalue 2 besides those ones given by the form of $(\phi,s)$ or in the language of Definition~\ref{def:extra}, are there any extra eigenfunctions?

\section{Extra eigenfunctions}\label{s:criterion}
In this section we are going to use tools from harmonic maps to study extra eigenfunctions.  
Suppose a metric $g$ has an extra eigenfunction, i.e. there exists a single-valued function $u$ in the Friedrichs domain with  $u\in\ker(\Delta-2)\setminus\{(\phi, s): s\in \bS^{2}\}$. To any such $u$ we associate a multivalued map $X\colon \Sigma_{\frakp}\to\mathbb{R}^3$ as follows
\begin{equation}\label{e:X}
X = u\phi + \frac{1}{|\phi_z|^2}(u_z\phi_{\bar z} + u_{\bar z}\phi_z),
\end{equation}
where $\phi$ is the multivalued harmonic map defined in~\eqref{e:phi} and $\Sigma_{\frakp}$ is defined as~\eqref{e:sigmap}.
\begin{proposition}\label{p:X1}
If $u$ is an extra eigenfunction and the multivalued map $X\colon \Sigma_{\frakp}\to\mathbb{R}^3$ is defined by~\eqref{e:X}, then $X$ satisfies the following properties:
\begin{itemize}
\item[1)] $X$ is non-constant;
\item[2)] $X$ is a conformal harmonic map, i.e. a (multivalued) branched minimal immersion;
\item[3)] $\phi$ is the Gauss map of $X$;
\item[4)] The monodromy groups of $X$ and $\phi$ coincide, i.e the support function $(X,\phi)$ is well-defined on $\Sigma_{\frakp}$. Moreover, the support function extends to a bounded function on $\Sigma$. 
\end{itemize}
\end{proposition}
\begin{remark}
If the developing map $f$ is a rational function i.e. the metric $g$ is induced by a branched cover of $\bS^{2}$, then the property $4)$ is usually formulated as "$X$ has flat ends", which means that $X$ has a well-defined tangent plane at infinity for each end of $X$.
\end{remark}

The proof of this proposition is almost identical to~\cite[p.156]{MR}. We present it here for completeness.  
\begin{proof}
Let $(\cdot,\cdot)$ denote the pointwise $\mathbb{C}$-{\bf bilinear} inner product  on $\mathbb{C}^3$ and let $\langle\cdot,\cdot\rangle$ be the usual Hermitian product. That is,
$$
(z,w)=zw, \ \langle z, w \rangle = z\bar w, \ \forall z,w\in \mathbb{C}^{3}.
$$
Since $\phi$ is conformal, we have $|\phi|^{2}=1$ and
\begin{equation}\label{e:phiz}
(\phi_{z}, \phi_{z})=\frac{1}{4} \left( (\phi_{x}, \phi_{x}) - (\phi_{y}, \phi_{y})\right) -2i (\phi_{x}, \phi_{y})=0.
\end{equation} 
Therefore
\begin{equation}\label{e:ortho}
\begin{aligned}
\langle \phi_{z}, \phi \rangle = \frac{1}{2}\partial_{z}|\phi|^{2}=0\\
\langle \phi_{\bar z}, \phi \rangle = \frac{1}{2}\partial_{\bar z}|\phi|^{2}=0\\
\langle \phi_{z}, \phi_{\bar z} \rangle = (\phi_{z}, \phi_{z})=0
\end{aligned}
\end{equation}
In other words, the vectors $\{\phi_z,\phi,\phi_{\bar z}\}$ form an orthogonal basis of $\mathbb{C}^3$ in $\langle\cdot,\cdot\rangle$. 

Since $\phi$ is a harmonic map, we also have
\begin{equation}\label{e:harm}
\phi_{\bar z z} = -|\phi_z|^2\phi.
\end{equation}
Combining~\eqref{e:phiz}, ~\eqref{e:ortho} and~\eqref{e:harm} we obtain
$$
(\phi_{zz},\phi) = -(\phi_z,\phi_z) = 0;\quad (\phi_{zz},\phi_z) = \frac{1}{2}(\phi_z,\phi_z)_z=0, 
$$
$$
(\phi_{zz},\phi_{\bar z}) = (\phi_z,\phi_{\bar z})_z - (\phi_z,\phi_{\bar z z}) = \partial_z|\phi_z|^2.
$$
As a result
\begin{equation}
\label{phizz:eq}
\phi_{zz} = \partial_z\ln|\phi_z|^2\phi_z.
\end{equation}
Therefore,
\begin{multline}\label{e:Xz}
X_z = u_{z}\left(\phi + \frac{\phi_{\bar z z}}{|\phi_z|^2}\right) + \left (u + \frac{u_{\bar z z}}{|\phi_z|^2}\right)\phi_{z} 
\\ +   u_{\bar z} \left(\partial_z\frac{1}{|\phi_z|^2} +\frac{\partial_z\ln |\phi_z|^2}{|\phi_z|^2}\right)\phi_z + F\phi_{\bar z}
\end{multline}
where
\begin{equation}\label{e:F}
F=\partial_z\left( \frac{u_z}{|\phi_z|^2} \right) = \frac{1}{|\phi_z|^2}(u_{zz} - \partial_z\ln|\phi_z|^2u_z).
\end{equation}
The first term in~\eqref{e:Xz} vanishes because of~\eqref{e:harm}, second term vanishes from~\eqref{e:eigen}, and one can check the vanishing of the third term by a direct computation. 
Therefore, we have 
\begin{equation}\label{e:Xz2}
X_z = F\phi_{\bar z}.
\end{equation}
In particular, by~\eqref{phizz:eq} $F=0$ if $u = (\phi,s)$. 

The identity $X_z = F\phi_{\bar z}$ also implies that $(X_{z}, X_{z})=F^{2}(\phi_{\bar z}, \phi_{\bar z})=0$, therefore $X$ is conformal. Additionally, $X_{\bar z} = \bar F\phi_z$, i.e. $\phi$ is the Gauss map of $X$. 

Furthermore, from~\eqref{phizz:eq} we have $\phi_{\bar z\bar z}=\tilde F \phi_{\bar z}$, and together with~\eqref{e:Xz2} this implies that $X_{z\bar z} = G\phi_{\bar z}$ where $G=F_{z}+\tilde F$. At the same time, $X_{z\bar z}$ is real, that is, $G\phi_{\bar z}=\bar G \phi_{z}$. However $\phi_{\bar z}$ and $\phi_{z}$ are orthogonal in $\mathbb{C}^{3}$, therefore $G=0$ and $X_{z\bar z}=0$.

If $X\equiv s \in \mathbb{R}^{3}$ were constant, then since $u = (X,\phi)$ one has $u = (\phi,s)$, which contradicts the definition of the extra eigenfunction.
\end{proof}

\begin{proposition}\label{p:X2}
Conversely, let $X$ be a map satisfying properties $1)-4)$. Then the support function $(X,\phi)$ is an extra eigenfunction.
\end{proposition}
\begin{proof}
The map $\phi$ is harmonic, i.e. $\phi_{\bar z z} = -|\phi_z|^2\phi$. And $\phi$ is a Gauss map of $X$, therefore $(\phi, X_z) = (\phi, X_{\bar z}) = 0$. Finally, $X$ is harmonic, i.e. $X_{\bar z z} = 0$. Using all these, we have
$$
0 = (X_z,\phi)_{\bar z} = (X_{z\bar z},\phi) + (X_z,\phi_{\bar z}) = (X_z,\phi_{\bar z});
$$
and 
\begin{multline}
(X,\phi)_{\bar z z} = \left((X_{\bar z},\phi) + (X,\phi_{\bar z})\right)_z = (X,\phi_{\bar z})_{z}\\ = (X_z,\phi_{\bar z}) + (X,\phi_{\bar z z}) = -|\phi_z|^2(X,\phi).
\end{multline}
Hence by~\eqref{e:eigen}, the function $u=(X,\phi)$ satisfies $\Delta_{g}u=2u$. And since $X$ is not constant, $u$ is an extra eigenfunction. 
\end{proof}

The multivalued map $X$ has different behaviour near a cone point depending on whether the point is integer or not. First notice that for a spherical metric $g$, the developing map $f$ (and hence its stereographic projection $\phi$) can be extended over a cone point with integer cone angle to be a single-valued meromorphic function nearby (cf~\cite[Lemma 4.3]{CWWX}). Therefore, for a metric $g$ with non-integer cone points $\{p_{1}, \dots, p_{j}\}$ and integer cone points $\{p_{j+1}, \dots, p_{k}\}$, it is convenient to introduce another notation of punctured surface (different from~\eqref{e:sigmap})
\begin{equation}\label{e:sigmastar}
\Sigma^{*}:=\Sigma\setminus\{p_{1}, \dots, p_{j}\}.
\end{equation}
From the discussion above,  $f$ and $\phi$ can be viewed as multivalued maps on $\Sigma^{*}$, and their monodromy in $\operatorname{PSU}(2)$ and $\operatorname{SO}(3)$ are defined correspondingly. 

We next use Weierstrass representation to reformulate the criterion of Propositions~\ref{p:X1} and~\ref{p:X2}. Define $\widetilde {\Sigma^{*}}$ as the universal cover of the punctured surface $\Sigma^*$ which is defined in~\eqref{e:sigmastar}, and fix a base point $z_{0}\in \widetilde{\Sigma^*}$. Recall that the Weierstrass representation of a function $X: \widetilde {\Sigma^{*}} \rightarrow \mathbb{R}^{3}$ is given by  
$$
X = \frac{1}{2}\Re\int (1-f^2,i(1+f^2),2f)\omega,
$$
or 
\begin{equation}\label{e:Xdef}
X(z)=X(z_{0}) + \frac{1}{2}\Re\int_{z_{0}}^{z} (1-f^2,i(1+f^2),2f)\omega,
\end{equation}
where $f: \widetilde {\Sigma^{*}}\rightarrow \overline{\mathbb{C}}$ is the lift of a developing map of $g$ (in other words, the Gauss map viewed as a meromorphic function), $\omega$ is a meromorphic $1$-form on $\widetilde{\Sigma^*}$ with poles only at the preimages of integer cone points $\{p_{j+1}, \dots, p_{k}\}$ and such that $f^2\omega$ is also meromorphic on $\widetilde{\Sigma^*}$ with poles only at the preimages of $\{p_{j+1}, \dots, p_{k}\}$. The quadratic differential defined as
 $$\sigma: = \omega df$$ 
 is the holomorphic part of the second fundamental form. It is metric invariant and, therefore, is a well-defined meromorphic quadratic differential on $\Sigma^*$. 

Let $\tilde \phi: \widetilde{\Sigma^{*}}\rightarrow \mathbb{R}^{3}$ be the lift of $\phi$.  The following theorem gives the necessary and sufficient conditions for $(X,\tilde \phi)$ to descend to a (bounded) $2$-eigenfunction on $\Sigma$.

\begin{theorem}\label{th1}
Let $g$ be a spherical conical metric, and let $f:  \widetilde {\Sigma^{*}} \rightarrow \overline{\mathbb{C}}$ be the lift of one of its developing maps. Then $g$ has an extra eigenfunction if and only if there exists a nonzero meromorphic one-form $\omega$ on $\widetilde{\Sigma^*}$ such that 
\begin{itemize}
\item[(a)] the meromorphic forms $\omega, f^{2}\omega$ have poles only at the preimages of integer cone points $\{p_{j+1}, \dots, p_{k}\}$; 
\item[(b)] the function $X$ defined in~\eqref{e:Xdef} satisfies the following condition for any non-trivial loop $\gamma$ on $\Sigma_\frakp$:
\begin{equation}\label{condition:eq}
\frac{1}{2}\Re\int^{z_{0}+\gamma}_{z_{0}} (1-f^2,i(1+f^2),2f)\omega = (A_\gamma-1)X(z_{0})
\end{equation}
where $A_{\gamma}$ is the $\operatorname{SO}(3)$ monodromy matrix of $\gamma$.
\item[(c)] the quadratic differential $\sigma:=\omega df$ is single-valued on $\Sigma^*$ and extends to a meromorphic quadratic differential on $\Sigma$ with at most a simple pole at any cone point $p_i$, $i=1,\ldots, k$.  
\end{itemize}
\end{theorem}

\begin{proof}
By Proposition~\ref{p:X1} and~\ref{p:X2}, we only need to show that condition (a) - (c) are satisfies if and only if there is a function $X: \widetilde {\Sigma^{*}} \rightarrow \mathbb{R}^{3}$ that satisfies properties 1) - 4) listed in Proposition~\ref{p:X1}. 

We first prove that if there exists such an $\omega$ then  the function $X$ constructed via the Weierstrass representation~\eqref{e:Xdef} satisfies 1) - 4).  Such $X$ is non-constant unless $\omega=0$. Since $f$ and $\omega$ is meromorphic we have $X_{z\bar z}=0$, and the location of poles implies that $X$ is a multivalued conformal harmonic map on $\Sigma_\frakp$. By construction $\tilde \phi$ is the Gauss map of $X$.   Therefore $X$ satisfies conditions $1)-3)$.

Now we check $X$ satisfies condition 4). We need to show that $X$, when viewed as a multivalued function on $\Sigma_\frakp$, has the same monodromy as $\phi$, therefore $(X,\tilde \phi)$ is a well-defined function on $\Sigma_\frakp$.  
Recall that $\sigma$ is a well-defined meromorphic differential on $\Sigma$ (i.e. with trivial monodromy), hence the $(1,0)$-differentials $\omega$ and $df$ have the opposite monodromy. 
%
On the other hand, the metric $g$ is well-defined on $\Sigma^{*}$, hence has trivial monodromy. Recall $g$ is defined as the pullback by $f$:
$$
g=\frac{4|df|^{2}}{(|f|^{2}+1)^{2}}.
$$
Therefore $df$ and $\frac{\overline{df}}{(|f|^{2}+1)^{2}}$ have the opposite monodromy, where  
$$\frac{\overline{df}}{(|f|^{2}+1)^{2}}=\frac{1}{(|f|^{2}+1)^{2}}\bar f_{\bar z} d\bar z$$
is a $(0,1)$-differential.
Combining the two statements above, we get that $\omega$ and $\frac{\overline{df}}{(|f|^{2}+1)^{2}}$ have the same monodromy. 
We then consider the monodromy of $X$ and $\phi$. Consider the $\mathbb{R}^3$-valued meromorphic differential $\partial X$ 
$$
\partial X=\frac{1}{4}\omega \left( (1-f^{2}), i (1+f^{2}), 2f\right)
$$
and the $\mathbb{R}^3$-valued $(0,1)$-differential $\bar \partial \phi$
$$
\bar \partial \phi=\frac{\overline{df}}{(|f^{2}|+1)^{2}}\left( (1-f^{2}), i (1+f^{2}), 2f \right).
$$
Comparing the two expressions we have that $\partial X$ and $\bar \partial \phi$ have the same monodromy.
Since the monodromy of $\phi$ belongs to $\operatorname{SO}(3)$, integration implies $\partial X$ has the same monodromy as $\phi$. 
At the time, the monodromy of $X$ itself is not necessarily in $\operatorname{SO}(3)$, therefore, it could have a translation component. Let $\gamma$ be the lift of any homotopically non-trivial loop on $\Sigma^*$ to $\widetilde {\Sigma^{*}}$ with base point $z_{0}$.  Then we have
\begin{equation}\label{e:Bgamma}
X(z_{0}+\gamma) = A_\gamma X(z_{0}) + B_\gamma.
\end{equation}
where $A_{\gamma}$ is the monodromy matrix for $\phi$ and $B_{\gamma}\in \mathbb{R}^{3}$.
To ensure that $B_\gamma=0$ we need to have
$$
 A_\gamma X(z_{0})=X(z_{0}+\gamma) =X(z_{0}) + \frac{1}{2}\Re\int^{z_{0}+\gamma}_{z_{0}} (1-f^2,i(1+f^2),2f)\omega 
$$
which is exactly~\eqref{condition:eq}.

Finally we need to show that $(X,\tilde \phi)$ extends to a bounded function $(X,\phi)$ on $\Sigma$. 
We now show that, if $\sigma$ has at most a simple pole at each cone point $p_i$, then the support function $(X, \phi)$ will be bounded. Away from cone points all functions are bounded, so the support function is bounded. So we restrict to a neighborhood of a cone point $p_{i}$ with angle $2\pi (\alpha+1)$. Since $(X,  \phi)$ is well-defined on $\Sigma_\frakp$, we can take any local representative of $X$ and $\phi$ for the computation. By~\cite{CWWX} it is possible to take a local coordinate such that the developing map is given by $f(z) = z^{\alpha+1}$. Then locally
\begin{equation}\label{e:Xexpansion}
\begin{split}
X&=\frac{1}{2}\Re\int (1-z^{2\alpha+2},i(1+z^{2\alpha+2}),2z^{\alpha+1})\frac{\sigma}{(\alpha+1)z^\alpha dz} \\
&=\frac{1}{2(\alpha+1)}\Re\int(z^{-\alpha} - z^{\alpha+2}, i(z^{-\alpha}+z^{\alpha+2}), 2z)\frac{\sigma}{dz}.
\end{split}
\end{equation}
Since $\lim_{z\rightarrow p_{i}}f(z) = 0$, one has $\lim_{z\rightarrow p_{i}}\phi(z) = (0,0,1)$. Therefore, in order to have a bounded function $u=(X,\phi)$, the last coordinate of $X$ has to be bounded. And this is true if $\sigma$ has a pole of order at most $1$ at $p_i$. Combining everything above we showed that condition (a) - (c) is sufficient for the existence of an extra eigenfunction. 

We now show the other direction of the theorem, that if there is an extra eigenfunction then there exists a meromorphic one-form $\omega$ that satisfies condition (a) - (c). By Proposition~\ref{p:X1} we can construct a multivalued function $X$ on $\Sigma_\frakp$ via~\eqref{e:X} which satisfied condition 1) - 4). Denote $\tilde X: \widetilde{\Sigma_\frakp} \rightarrow \mathbb{R}^{3}$ to be the lift of $X$. Since $X$ (hence $\tilde X$) is conformal harmonic, there exists a Weierstrass representation given by~\eqref{e:Xdef} with $\omega$ being a meromorphic one-form on $\widetilde{\Sigma^*}$. Since $X$ has the same monodromy as $\phi$, it is straightforward that~\eqref{condition:eq} holds. By the same computation as above, $(X, \phi)$ is bounded on $\Sigma$ if and only if $\sigma$ has at most simple poles at any cone point. Together this implies that $\omega$ satisfies condition (a) - (c).
\end{proof}

We remark that condition (b) stated in the theorem must hold for any closed curve $\gamma$ on $\Sigma_\frakp$. In particular, if $A_\gamma=1$, then there is no condition on $X(z_{0})$, but the integral on the left hand side of~\eqref{condition:eq} must be zero.
\begin{proposition}
\label{integerangle:prop}
 If $\gamma$ is the path around a cone point $p_i$ with integer cone angle, then $~\eqref{condition:eq}$ holds if and only if $\Res_{p_i}\frac{\sigma}{df}=0$. 
\end{proposition}
\begin{proof}
We can see that $z^{-\alpha} \sigma/dz$ vanishes in the first two terms in the local expansion of $X$ in~\eqref{e:Xexpansion} and the third term is 0 if and only if $\Res_{p_i}\frac{\sigma}{df}=0$.
\end{proof}

On the other hand, if $A_\gamma \ne 1$, then~\eqref{condition:eq} gives a particular choice of the ``initial value'' $X(z_{0})$.
Moreover, if there are two paths $\gamma\ne\gamma'$ with non-trivial monodromy, then~\eqref{condition:eq} yields a compatibility condition, i.e. there has to be an initial value that works for all such paths simultaneously.

As the first application of the theorem above, we have the following observation:

\begin{corollary}
\label{noextra:cor}
There are no extra eigenfunctions for spherical conical metrics on $\mathbb{S}^{2}$ with at most three conical points.
\end{corollary}
\begin{proof}
The meromorphic quadratic differential $\sigma$ has at most simple poles at cone points. However by Riemann--Roch Theorem, a meromorphic quadratic differential on $\mathbb{S}^{2}$ should have at least four such points with at most simple poles.
\end{proof}
\begin{remark}
From the existing results~\cite{Troyanov, Eremenko, UY}, it was conjectured that any irreducible spherical metrics on $\bS^{2}$ with two or three conical metrics has no eigenfunction with eigenvalue 2. Our result gives a proof of this conjecture.
\end{remark}

In practice, condition~\eqref{condition:eq} is difficult to verify. However, if there are only a few non-integer points, then Proposition~\ref{integerangle:prop} makes it more manageable. Indeed, it is sufficient to check~\eqref{condition:eq} for generators of $\pi_1(\Sigma_\frakp)$, e.g. for simple loops around all but one cone points. There are no conical metrics  on $\bS^{2}$ with a single non-integer point, see e.g.~\cite{Troyanov, MP}. On the other hand, if there are exactly 2 non-integer points and the condition of Proposition~\ref{integerangle:prop} is satisfied for all integer points, then we only need to look for the right choice of $X(z_0)$ for~\eqref{condition:eq} to be satisfied. This allows to explicitly find developing maps admitting extra eigenfunctions.
\begin{remark}
\label{rmk:rescaling}
In the case that there are exactly two non-integer points, we have the following fact. If $f(\cdot): \Sigma^{*}\rightarrow \overline{\mathbb{C}}$ is a developing map that gives a metric with extra eigenfunctions, then for any $\lambda\in \mathbb{C}\setminus \{0\}$, $f(\lambda\cdot)$ and $\lambda f(\cdot)$ are developing maps for metrics with extra eigenfunctions as well. For $f(\lambda\cdot)$, it is just a reparametrization of the same metric, and $\omega(\lambda \cdot)$ would be the corresponding one-form that satisfies all the conditions. For $\lambda f(\cdot)$, the same $\omega$ for $f(\cdot)$ would satisfy all the conditions; however it generates a different metric and different eigenfunctions.  
\end{remark}
\begin{example}
\label{ex1}
Let $\Sigma=\bS^{2}\simeq \overline{\mathbb{C}}$. We would like to find a reducible metric $g$ with two non-integer points $0, \infty$ such that $g$ has extra eigenfunctions. Its developing map $f$ on $\Sigma^*=\mathbb{C}\setminus \{0\}$ is of the form
$$
f(z) = \sqrt{z}\frac{az+b}{cz+d}.
$$
If $g$ admits extra eigenfunctions, then the reducible metrics with developing maps $\lambda f(z)$ and $f(\lambda z)$, $\lambda\in \mathbb{C}\setminus \{0\}$ do as well, thus it is sufficient to set 
$$
f(z) = \sqrt{z}\frac{z+(b+1)}{z+1} = \sqrt{z}\left(1 +\frac{b}{z+1}\right),
$$
where $b\ne 0$.
The corresponding metric $g$ has two non-integer cone points at $0,\infty$ both with cone angle $\pi$. One has
$$
f'(z) = \frac{z^2 - (b-2)z + (b+1)}{\sqrt{z}(z+1)^2},
$$
thus, the metric $g$ also has two integer cone points with angle $4\pi$ at 
$$
z_\pm = \frac{1}{2}(b-2\pm \sqrt{b(b-8)}).
$$
Thus, one can additionally assume $b\ne 8$. Otherwise there are only $3$ cone points and by Corollary~\ref{noextra:cor} the map $f$ does not admit extra eigenfunctions. 

The meromorphic quadratic differential $\sigma$ has a simple pole at each cone point of $g$. There are $4$ such points and by Riemann-Roch, the space of such meromorphic quadratic differentials is one-dimensional, therefore,
$$
\sigma = \frac{\lambda}{z(z^2 -(b-2)z + (b+1))}\,dz^2.
$$
We aim to find $b\ne 0,8$ such that the condition of Proposition~\ref{integerangle:prop} is satisfied, i.e. that $\Res_{z=z_\pm}\frac{\sigma}{df} = 0$. The explicit computation shows that
$$
\Res_{z=z_\pm}\frac{\sigma}{df} = -\frac{\lambda(z_\pm+1)}{4\left(\sqrt{z_\pm}(z_\pm-z_\mp)\right)^3}(b-4)(b\pm \sqrt{b(b-8)}) 
$$
Since $z_\pm\ne -1$ for $b\ne 0$ one has $b=4$. Therefore, 
$$
f(z) = \sqrt{z}\frac{z-3}{z+1}
$$
admits extra eigenfunctions.

In terms of geometric configuration, this metric is realized by gluing three footballs, see Figure~\ref{f:3ftb}. Note that this cone angle combination, by \cite{MP}, can only be realized by reducible metrics. 

\begin{figure}[htbp]
\begin{center}
\includegraphics[width=0.5\textwidth]{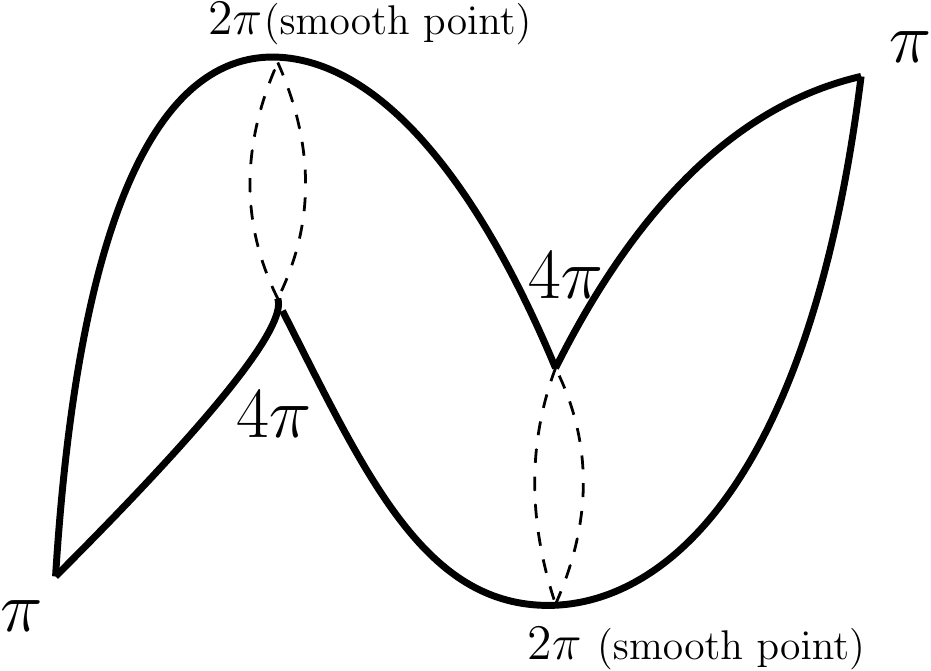}
\caption{The reducible conical metric given by the developing map $f(z) = \sqrt{z}\frac{z-3}{z+1}$. It can be obtained by gluing three spherical footballs of angle $\pi$, and has four conical points with angle $\pi, \pi, 4\pi, 4\pi$.}
\label{f:3ftb}
\end{center}
\end{figure}

\end{example}

Let us make two remarks about this example. First, one could do a similar computation with $f(z) = z^\beta\frac{az+b}{cz+d}$, $0<\beta<1$. Second, one could use the formula $(X,\widetilde \phi)$ to find the explicit formula for the extra eigenfunctions, but that would involve integration by~\eqref{e:Xdef}. We do not do these computations, because the next section contains a different construction which is more suitable for calculations. In particular, the formula for extra eigenfunctions is purely algebraic, and can be obtained below. 

%
\begin{example}
\label{intro:example}
We show in Section~\ref{examples:sec} that 
for any $\beta>0$, $k>\beta$, $k\in\mathbb{Z}$, the following developing map 
$$
f=z^\beta\frac{(k-\beta)z^k - (k+\beta)}{z^k+1}
$$ 
gives a spherical conical metric which has 2 extra eigenfunctions.

One can check that there are two extra eigenfunctions, given by real and imaginary parts of the function $h_{1}/h_{2}$, where
\begin{multline}
h_{1}=(|z|^{2k} + z^k - \bar z^k -1) \\+|z|^{2\beta}\left((k+\beta)^2 -(k^2-\beta^2)\bar z^k+(k^2-\beta^2)z^k-(k-\beta)^2|z|^{2k}\right)
\end{multline}
and
\begin{multline}
h_{2}=(|z|^{2k} + z^k + \bar z^k +1)\\+|z|^{2\beta}((k+\beta)^2 -(k^2-\beta^2)\bar z^k-(k^2-\beta^2)z^k+(k-\beta)^2|z|^{2k}).
\end{multline}
It follows from a standard lengthy computation, which can be done using a computer, that indeed the above expressions give two extra eigenfunctions.

Finally, we remark that for $\beta\in\mathbb{Z}$ it was previously known that $f$ has extra eigenfunctions, see e.g.~\cite[formula (4.2)]{Ejiri0}. 

\end{example}



\section{Deformations of harmonic maps}\label{s:deform}

In this section we take a different point of view on constructing metrics with extra eigenfunctions. Recall that for maps constructed in Section~\ref{s:criterion} the formula for extra eigenfunctions involves integration. This could cause difficulties if one tries to compute the eigenfunctions explicitly. The advantage of the method described in the present section is that the formula for extra eigenfunctions is purely algebraic. 
In order to explain the main idea let us assume for a moment that all angles are integer, i.e. the developing map $f$ is a branched covering $\Sigma\to\mathbb{S}^2$. 

Let $g_0$ be a fixed metric on $\Sigma$ and $[g_0]$ is the corresponding conformal class of metrics. Recall that a map $\Phi\colon(\Sigma,[g_0])\to\mathbb{S}^n$ is harmonic iff
$$
\Delta_g\Phi = |d\Phi|_g^2\Phi
$$
for any metric $g$ conformal to $g_0$, where $|d\Phi|_g^2 =\sum_{i=1}^{n+1} |d\Phi^i|^2_g$. By conformal covariance of $\Delta$, the components $\Phi^i$ are eigenfunctions of $\Delta_{g_\Phi}$ with eigenvalue $2$, where $g_\Phi = \frac{1}{2}|d\Phi|_g^2g$. The metric $g_\Phi$ has conical singularities of integer angles at points where $d\Phi$ vanishes. It is well-known that all such harmonic maps are (weakly) conformal, i.e. $g_\Phi = \Phi^*g_{\mathbb{S}^n}$. 
\begin{definition}
We will call a smooth map $\Phi\colon\Sigma\to\mathbb{S}^n$ {\em linearly full} if its image linearly spans $\mathbb{R}^{n+1}$ or, equivalently, if the coordinate functions $\Phi^i$ form an $(n+1)$-dimensional space of functions on $\mathbb{S}^2$. 
\end{definition}

We now explain the connection to conical spherical metrics with extra eigenfunctions. Suppose that one has a smooth family $\Phi_t$ of harmonic maps $\Phi_t\colon\Sigma\to\mathbb{S}^n$ such that
\begin{itemize}
\item[1)] For $t\ne 0$ the map $\Phi_t$ is linearly full;
\item[2)] For $t=0$ the map $\phi = \Phi_0$ is a branched covering over an equatorial $\mathbb{S}^2\subset\mathbb{S}^n$. 
\end{itemize}
Then $g_\phi$ is a spherical conical metric with extra eigenfunctions. Indeed, for $t\ne 0$ one has $\ker(\Delta_{g_{\Phi_t}}-2)\geqslant n+1$. Therefore, by continuity of eigenvalues the same is true for $g_\phi$, i.e. $\ker(\Delta_{g_\phi}-2)\geqslant n+1$. 

The construction of such families $\Phi_t$ has been studied, see for example~\cite{Ejiri,Kotani}. Below we describe this construction in detail and adapt it to the setting of multivalued harmonic maps in order to obtain examples of conical metrics with extra eigenfunctions and non-integer angles.

\subsection{Single valued totally isotropic harmonic maps to $\mathbb{S}^{2m}$.} 
\label{TIHSM:sec}
In this section we briefly describe a construction of the particular class of harmonic maps from surfaces to $\mathbb{S}^{2m}$, called {\em totally isotropic}.  Recall that a subspace $L$ of $\mathbb{C}^{2m+1}$  is called {\em isotropic} if $L\perp\bar L$ with respect to the usual Hermitian product.  Equivalently, this means the $\mathbb{C}$-bilinear inner product $(\cdot,\cdot)$ is identically $0$ on $L$.
 \begin{definition}
 \label{TIHM:def}
 A harmonic map $\Phi\colon \Sigma\to \mathbb{S}^{2m}$ is called {\em totally isotropic} if, for any point $p\in M$, the space 
 $$\Psi(p):=\mathrm{span}_{\mathbb{C}}\{\Phi_{\bar z}(p),\Phi_{\bar z \bar z}(p),\ldots,\Phi^{(m)}_{\bar z}(p)\}$$ 
 is isotropic.
 \end{definition} 
\begin{remark}
All harmonic maps $\mathbb{S}^2\to\mathbb{S}^{n}$ are totally isotropic (see~\cite{Calabi}), so this construction completely describes such maps. 
\end{remark}

Totally isotropic maps can be constructed using complex geometry via the so-called {\em twistor correspondence} which can be seen as an extension of the fact that all harmonic maps $\mathbb{S}^2\to\mathbb{S}^2$ are holomorphic. 
\begin{definition}
We define the {\em twistor space} $\mathcal H_m$ to be the space of isotropic $m$-planes in $\mathbb{C}^{2m+1}$. 
\end{definition}
The space $\mathcal H_m$ is a complex submanifold of the complex Grassmanian $Gr_{m,2m+1}(\mathbb{C})$, thus $\mathcal H_m$ is K\"ahler. We then define the {\em twistor projection}
$$\begin{aligned}
\pi_m\colon&\mathcal H_m\to\mathbb{S}^{2m}\\
&L \mapsto (L\oplus\bar L)^\perp
\end{aligned}
$$
for an appropriate choice of orientation on $\mathbb{S}^{2m}$, where $(L\oplus\bar L)^\perp$ is viewed as a real subspace of $\mathbb{C}^{2m+1}$ hence its orthogonal complement is well-defined on $\mathbb{S}^{2m}$. 

The vertical distribution on $\mathcal H_m$ is given by $\ker d\pi_m$, and horizontal distribution is given by $(\ker d\pi_m)^\perp$. The map $\pi_m$ is a Riemannian submersion, i.e. it preserves the length of horizontal vectors. 
\begin{definition}
A smooth map from a manifold $M$ to $\mathcal H_m$ is called {\em horizontal} if the image of the differential is inside the horizontal distribution.
\end{definition}

\begin{proposition}[Twistor correspondence, see e.g.~\cite{Barbosa}]
\label{twistor:prop}
Let $\Sigma$ be a Riemann surface and $\Sigma\colon M\to \mathbb{S}^{2m}$ be a linearly full totally isotropic harmonic map. Then outside of   a collection of isolated points the map $\Psi$ from Definition~\ref{TIHM:def} defines a horizontal  holomorphic map to $\mathcal H_m$. Furthermore, $\Psi$ can be extended across these isolated points to yield a horizontal holomorphic map $\Psi\colon \Sigma\to\mathcal H_m$. Conversely, if $\Psi\colon \Sigma\to\mathcal H_m$ is a holomorphic horizontal map, then $\Phi=\pi_m\circ \Psi$ is a harmonic map to $\mathbb{S}^{2m}$. The map $\Psi$ is called a {\em twistor lift} of $\Phi$.
\end{proposition} 
\begin{remark} Let $m=1$. Then by the third equality in~\eqref{e:ortho} any conformal harmonic map $\phi$ to $\mathbb{S}^2$ is totally isotropic. In notation of Section~\ref{s:SphCo} the twistor lift can be identified with the function $f$.
\end{remark}

Let us also recall another object associated with a map $\Phi$ and its twistor lift $\Psi$, called {\em directrix} of $\Phi$.

\begin{definition}
Let $\Phi\colon \Sigma\to \mathbb{S}^{2m}$ be a linearly full totally isotropic harmonic map. Then the {\em directrix} of $\Phi$ is a holomorphic map $\Xi\colon M\to \mathbb{CP}^{2m}$ such that outside of a collection of isolated points 
$$\Psi = \mathrm{span}_{\mathbb{C}}\{\xi,\xi_{z},\ldots,\xi^{(m-1)}_{z}\}$$ 
where $\xi$ is any local lift of $\Xi$ to $\mathbb{C}^{2m+1}$.
\end{definition} 
For a given linearly full totally isotropic harmonic map, an explicit construction of the directrix $\Xi$ is given in~\cite{Barbosa}. Conversely, given a holomorphic map $\Xi\colon \Sigma\to \mathbb{CP}^{2m}$ such that outside of a collection of isolated points the map $\Psi = \mathrm{span}_{\mathbb{C}}\{\xi,\xi_{z},\ldots,\xi^{(m-1)}_{z}\}$ has its image in $\mathcal H_m$, the map $\Psi$ can be extended to a horizontal holomorphic map, 
so that $\Xi$ is a directrix curve of $\Phi = \pi_m\circ\Psi$.  The directrix is often used in order to provide examples of totally isotropic maps, since it is usually easier to construct a map to $\mathbb{CP}^m$ rather than a map to $\mathcal H_m$. We will use this below in Section~\ref{examples:sec}.

We now recall the construction of extra eigenfunctions in the case of linearly full totally isotropic harmonic maps. There is a natural holomorphic action of $\operatorname{SO}(2m+1,\mathbb{C})$ on $\mathcal H_m$ that preserves horizontal distributions. This induces an action on totally isotropic harmonic maps: if $A\in \operatorname{SO}(2m+1,\mathbb{C})$ and $\Phi = \pi_m\circ \Psi$, then $A\Psi$ is a holomorphic horizontal map, so one can define $A\Phi:=\pi_m\circ A\Psi$. One then has the following:

\begin{theorem}[Ejiri~\cite{Ejiri}, Kotani~\cite{Kotani}]
\label{S2deformation:thm}
Let $\Phi\colon (\Sigma,g)\to \mathbb{S}^{2m}$, $m>1$ be a linearly full totally isotropic harmonic map. Then there exists a one-parameter subgroup $A_t\subset \operatorname{SO}(2m+1,\mathbb{C})$ such that $A_t\Phi\to \Phi_\infty$ in $C^\infty$ topology, where $\Phi_\infty\colon (\Sigma,g)\to \mathbb{S}^{2m}$ is totally isotropic, but not linearly full, i.e. its image is contained in a proper totally geodesic subsphere $\mathbb{S}^{2m'}$, $m'<m$. If in addition $m'=1$, then the metric $g_{\Phi_\infty}$ has $2(m-m')$ extra eigenfunctions.
\end{theorem}
We provide the sketch of a proof, similar consideration are used later in the context of conical metrics with non-integer angles.
\begin{proof}[Sketch of the proof]
Let $\{e_0,e_1,\ldots, e_{2m}\}\subset\mathbb{R}^{2m+1}\subset\mathbb{C}^{2m+1}$ be a real orthonormal basis. We then define $E_j = \frac{1}{\sqrt{2}}(e_{2j-1} - ie_{2j})$, $j=1,\ldots, m$. Then the basis $\{e_0, E_1,\bar E_1,\ldots E_m, \bar E_m\}$ satisfies 
$$(E_j,E_k) = 0, \ (E_j,\bar E_k) = \delta_{jk}, \ (e_0, E_j) = (e_0, \bar E_j) = 0$$
where $(\cdot,\cdot)$   is the $\mathbb{C}$-bilinear inner product. 

Let $A_t\in \operatorname{SO}(2m+1,\mathbb{C})$ be a family of orthogonal transformations given by 
$$A_tE_1 = e^t E_1, \ A_t\bar E_1 = e^{-t}\bar E_1, \ A_t= id \text{ on } \mathrm{span}_{\mathbb{C}}\{E_1,\bar E_1\}^\perp.
$$ 
Recall that $A_{t}$ acts on $\mH_m$. 
One can show that the fixed point set of $A_t$ on $\mH_m$ consists of two components 
$$\mF_+ = \{E_1\oplus L'\}, \ \mF_-=\{\bar E_1\oplus L'\},$$ 
where $L'$ is an $(m-1)$-dimensional isotropic subspace in $\mathrm{span}_{\mathbb{C}}\{E_1,\bar E_1\}^\perp$. The space of all such $L'$ can be identified with $\mH_{m-1}$ and the twistor projection $\pi_m$ sends these fixed point sets to the subsphere $\mathbb{S}^{2(m-1)} = \mathbb{S}^{2m}\cap \mathrm{span}_{\mathbb{R}}\{e_1,e_2\}^\perp$. Furthermore, for any $L\in \mH_m\setminus \mF_-$ one has $A_t L\to L_{+\infty}\in \mF_+$ as $t\to +\infty$; and vice versa, for any $L\in \mH_m\setminus \mF_+$ one has $A_t L\to L_{-\infty}\in \mF_-$ as $t\to -\infty$. These facts can be proved either by a direct computation or by noticing that $A_t$ is the gradient flow of a Morse--Bott function (see e.g.~\cite{GuestOhnita}).

Given a linearly full totally isotropic harmonic map $\Phi\colon (\Sigma,g)\to\mathbb{S}^{2m}$ with the twistor lift $\Psi$, we assume that $\Psi(\Sigma)\subset \mH_m\setminus \mF_+$ or $\Psi(M) \subset \mH_{m}\setminus \mF_-$. Then $A_t\Psi$ as $t\to+\infty$ or $t\to-\infty$ respectively converges to a holomorphic horizontal map $\Psi_{\pm\infty}\subset \mF_{\pm}$, such that $\Phi_\infty = \pi_m\circ\Psi_\infty$ is not linearly full.

Notice that the sets $\mF_+$, $\mF_-$ depend on the choice of the basis $\{e_{i}\}_{i=0}^{2m}$. We show that for any linearly full totally isotropic harmonic map $\Phi$, one can choose the basis, so that $\Psi(\Sigma)$ avoids $\mF_+$. This is achieved by dimension counting. The image $\Psi(\Sigma)$ intersects $\mF_+$ iff there is a point $x\in M$ such that the line $\mathbb{C}(e_1+ie_2)$ is a subspace in $\Psi(x)$. Let $\mA = \{l\in \mathbb{CP}^{2m},\,l\subset \Psi(x)|\,x\in \Sigma \}$ be a set of isotropic lines contained in $\Psi(x)$. It is a complex variety of (complex) dimension at most $m$. At the same time, any isotropic line is spanned by a vector of the form $f_1+i f_2$, where $f_1\perp f_2$, $\|f_1\|=\|f_2\|=1$ are real vectors. Therefore $f_1,f_2$ can be chosen to be a part of a real orthonormal basis. The set of all isotropic lines is a smooth quadric in $\mathbb{CP}^{2m}$, i.e. it has (complex) dimension $2m-1$. Since $2m-1>m$ for $m>1$ the proof is complete.
\end{proof}

Theorem~\ref{S2deformation:thm} allows one to construct spherical conical metrics $g_{\Phi_\infty}$ with integer angles and arbitrary many extra eigenfunctions~\cite{Ejiri}. The argument in~\cite{Ejiri} is inductive. Starting with a linearly full map $\Phi\colon(\Sigma,g)\to\mathbb{S}^{2m}$ one first applies Theorem~\ref{S2deformation:thm} to obtain a linearly full map $\Phi_\infty\colon (\Sigma,g)\to\mathbb{S}^{2m'}$, $m'<m$, with extra eigenfunctions. Then one applies Theorem~\ref{S2deformation:thm} to $\Phi_\infty$ again and again until one arrives at a map to $\mathbb{S}^2$. In this approach it is necessary to track how the action of $\operatorname{SO}(2m+1,\mathbb{C})$ affects the extra eigenfunctions to show that they do not completely disappear after each step.

In fact, it is possible to simplify this argument by completely avoiding the latter issue. It is sufficient to construct a one-parameter subgroup $A_t$ in the proof of Theorem~\ref{S2deformation:thm} in a way that ensures $m'=1$. We sketch the argument below. Let $\{e_0, E_1,\ldots, E_m\}$ be the vectors given above. Let $A_t$ be given by 
$$
\begin{aligned}
&A_tE_j = e^{-t} E_j, \ A_t\bar E_j = e^{t}\bar E_j, \ j=1,\ldots, m-1\\
&A_t=id \text{ on } \mathrm{span}_{\mathbb{C}}\{e_0,E_m,\bar E_m\}.
\end{aligned}
$$ 
Denote $L_0' = \mathrm{span}_\mathbb{C}\{E_1,\ldots,E_{m-1}\}$ as the $e^{-t}$-eigenspace of $A_t$. Then one observes that if $L\in\mH_m$ satisfies $L\cap L_0'=\{0\}$, then as $t\to +\infty$
$$A_t L\to \bar L'_0\oplus l$$ 
where $l$ is an isotropic line in $\mathrm{span}_{\mathbb{C}}\{e_0,E_m,\bar E_m\}$. The twistor projection $\pi_m$ sends such planes into $\mathbb{S}^2 = \mathrm{span}_{\mathbb{R}}\{e_0,e_{2m+1},e_{2m}\}\cap \mathbb{S}^{2m}$. Thus, it is sufficient to choose the vectors $E_i$ such that $\forall p\in M,\, \Psi(p)\cap L_0' = \{0\}$. This can be achieved by dimension counting similar to the proof of Theorem~\ref{S2deformation:thm}. The space of $(m-1)$-dimensional isotropic planes is a complex manifold of (complex) dimension $\frac{m^2+3m-4}{2}$. At the same time, 
\begin{multline*}
\mA = \{L'|\,L'\cap \Psi(p)\ne \{0\},\,p\in \Sigma,\\ \text{$L'$ is an $(m-1)$-dimensional isotropic plane}\}
\end{multline*}
 is a variety of dimension $1+(m-1) + \frac{(m-1)^2 +3(m-1)-4}{2} = \frac{m^2+3m-6}{2}<\frac{m^2+3m-4}{2}$, which completes the argument. 

Let us discuss the extra eigenfunctions of such maps. 
Given a linearly full totally isotropic map $\Phi\colon(\Sigma,g)\to\mathbb{S}^{2m}$, $m>1$ and its twistor lift $\Psi\colon \Sigma\to\mH_m$,  let $L'$ be an isotropic $(m-1)$-plane such that $\forall p\in M$ one has $\Psi(p)\cap L'=\{0\}$. Choose a basis $E_1,\ldots E_{m-1}$ of $L'$ such that $(E_k,\bar E_j) = \delta_{kj}$. Let $e_1,e_2,e_3$ be a real orthonormal basis of $V=(L'\oplus \bar L')^\perp$, Then by the result of~\cite{Kotani} one can choose a basis of $\Psi(z)$ of the form
\begin{equation}
\label{specialbasis:eq}
\begin{split}
&F_j(z) = \bar E_j + G_j(z) + u_j(z),\,j=1,\ldots, m-1;\\
&F_m(z) = w(z) + u_m(z),
\end{split}
\end{equation}
where $w,G_j\colon \Sigma\to V$, $u_j\colon \Sigma\to L'$ are meromorphic. Then $\mathbb{C}w$ is a twistor lift of $\Phi_\infty$. That is, if we write $w(z) =\sum_{i=1}^{3} w^i(z)e_i$, then the developing map $f(z)$ (which is $\Phi_\infty$ viewed as a map to $\overline{\mathbb{C}}$) can be found using the formula 
\begin{equation}\label{e:develop}
f(z) = \frac{w^3(z)}{w^1(z)-iw^2(z)}.
\end{equation}
Finally, the functions 
\begin{equation}\label{e:eigenfunctions}
h_j(z) = (G_j(z),\Phi_\infty)
\end{equation}
are the complex valued extra eigenfunctions. 

\begin{example} 
Let $\Sigma = \mathbb{S}^2$ viewed as $\overline {\mathbb{C}}$. Then it is easy to check that the curve $\Xi\colon\mathbb{S}^2\to\mathbb{CP}^{2m}$ with the lift $\xi(z)$ given by
\begin{equation*}
\xi(z) = \frac{1}{6}\bar E_1 - \frac{2}{3}z\bar E_2 + z^2e_0 + z^{3} E_2 + z^{4} E_1
\end{equation*}
is a directrix curve.
Thus $\Psi(z)=\mathrm{span}\{\xi(z),\xi'(z)\}$ is a holomorphic horizontal map to $\mH_2$, where
\begin{equation*}
\begin{cases}
&\xi(z) =\frac{1}{6}\bar E_1 - \frac{2}{3}z\bar E_2 + z^2e_0 + z^{3} E_2 + z^{4} E_1;\\
&\xi'(z) = -\frac{2}{3}\bar E_2 + 2z e_0 + 3z^{2} E_2 + 4z^{3} E_1.
\end{cases}
\end{equation*}
One has $\Psi(0) = \bar E_1\oplus \bar E_2$ and $\Psi(\infty) = \overline{\Psi(0)}$, therefore, one can not choose $L'$ to be spanned by any of the elements of the basis. Thus, it is convenient to do a change of basis, e.g.
\begin{equation*}
\begin{cases}
&\bar E_1 = \frac{1}{\sqrt{2}}e'_0 -\frac{1}{2}(E_1' - \bar E_1'),\\
& E_1 = \frac{1}{\sqrt{2}}e'_0 + \frac{1}{2}(E_1' - \bar E_1'),\\
&e_0 = \frac{1}{\sqrt{2}}(E_1' + \bar E_1'),
\end{cases}
\end{equation*}
and pick $L' = \mathbb{C}E_1$. Using the formulae~\eqref{specialbasis:eq},~\eqref{e:develop},~\eqref{e:eigenfunctions} together one obtains that
$$
f(z) = z\frac{z^2 + \frac{1}{2}}{z^2 - \frac{3}{\sqrt{2}}}
$$
is the developing map of $\Phi_\infty$. Up to transformations in Remark~\ref{rmk:rescaling} we recover the case $\beta =1$, $k=2$ of Example~\ref{intro:example}.
\end{example}

In Example~\ref{ex:m=2} below we generalize this approach to the case of metrics with non-integer angles.

\subsection{Multivalued totally isotropic harmonic maps to $\mathbb{S}^{2m}$} 
\label{CTIHMM:sec}
As before, let $\Sigma^* = \Sigma\setminus \{p_1,\ldots, p_k\}$ and let $\widetilde{\Sigma^{*}}$ be its universal cover.
\begin{definition}
\label{CTIHMM:def}
A multivalued map $\Phi\colon\Sigma^*\to\mathbb{S}^{2m}$ is called {\em conical totally isotropic harmonic} (or {\em CTIHMM}) if 
\begin{itemize}
\item[1)] the lift $\widetilde\Phi\colon\widetilde{\Sigma^{*}}\to\mathbb{S}^{2m}$ is a totally isotropic harmonic map;
\item[2)] the monodromy of $\Phi$ is contained in $\operatorname{SO}(2m+1,\mathbb{R})$; 
\item[3)] the metric $g_\Phi$ extends to a metric with conical singularities on $\Sigma$. That is, near any $p_j$ one has a local expression 
$$|d\Phi|_g^2(z) = |\rho(z)|^2 |z|^{2(\alpha_j-1)}$$ 
where $z$ is a coordinate centered at $p_j$, $\rho(z)\ne 0$ and $\alpha_j>0$.
\end{itemize}
\end{definition}
\begin{remark}
The energy density $|d\Phi|_g^2$ is a well-defined function on $\Sigma^*$ by item 2), therefore it is valid to discuss its local expressions near any punctures.
\end{remark}

Since the proof of Proposition~\ref{twistor:prop} is purely local, any CTIHMM $\Phi\colon\Sigma^*\to \mathbb{S}^{2m}$ defines a horizontal holomorphic multivalued twistor lift $\Psi\colon\Sigma^*\to\mathcal H_m$ with the same monodromy representation. The latter can be seen directly from the explicit formula for $\Psi$ given in Definition~\ref{TIHM:def}.

We are now ready to study the action of $\operatorname{SO}(2m+1,\mathbb{C})$ on CTIHMMs. Let $\Phi$ be a linearly full CTIHMM and $\Psi$ be the corresponding twistor lift. If $A\in \operatorname{SO}(2m+1,\mathbb{C})$ and $M$ is the monodromy group of $\Psi$, then the monodromy group of $A\Psi$ is given by $AMA^{-1}$. In particular, we already see the difference with the situation described in Section~\ref{TIHSM:sec}, namely, the twistor projection $\pi_m\circ A\Psi$ is {\em not} necessarily a CTIHMM. The necessary condition is that $AMA^{-1}\subset \operatorname{SO}(2m+1,\mathbb{R})\subset \operatorname{SO}(2m+1,\mathbb{C})$. It turns out that this condition is also sufficient.

\begin{proposition}
Let $\Phi$ be a linearly full CTIHMM with monodromy group $M$ and $\Psi$ be the corresponding twistor lift. Then for any $A\in \operatorname{SO}(2m+1,\mathbb{C})$ satisfying $AMA^{-1}\subset \operatorname{SO}(2m+1,\mathbb{R})\subset \operatorname{SO}(2m+1,\mathbb{C})$, the map 
$\pi_m\circ A\Psi$ is a linearly full CTIHMM with the same conical singularities.
\end{proposition}
\begin{proof}
One needs only to check that the item $3)$ of Definition~\ref{CTIHMM:def} is satisfied for $A\Psi$. To do that we recall the following facts for $\mH_m$.

Since $\mH_m\subset Gr_{m,2m+1}(\mathbb{C})$ is a complex submanifold of a K\"ahler manifold, a metric on $\mH_m$  is given by the restriction of the metric on $Gr_{m,2m+1}(\mathbb{C})$ and can be described in the following way. The holomorphic tangent space $T^{(1,0)}_L\mH_m$ can be identified with a subspace of $\mathrm{hom}_{\mathbb{C}}(L,L^\perp)$. The metric on this space is the usual $(X,Y):=\mathrm{tr}(X^*Y)$, so that the length of a tangent vector coincides with the Hilbert-Schmidt norm of the corresponding operator.

Let $X\in T^{(1,0)}_L\mH_m\subset \mathrm{hom}_{\mathbb{C}}(L,L^\perp)$, $A\in \operatorname{SO}(2m+1,\mathbb{C})$. We denote by $A_*$ the differential of the action of $A$ on $\mathcal H_m$, i.e. $A_*X\in T^{(1,0)}_{AL}\mH_m\subset \mathrm{hom}_{\mathbb{C}}(AL,(AL)^\perp)$. Then one has that $A_*X = \pi^\perp_{AL}AXA^{-1}$, where $\pi^\perp_{AL}$ is the projection onto $(AL)^\perp$ (see e.g.~\cite[Section 4.6]{KarpukhinRP2}). Since the norm of a projection equals the dimension of the image, one has  $\|A_*X\|\leqslant c\|X\|$, where $c = (m+1)\|A\|\|A^{-1}\|$ is a constant independent of $X$. Note that $(A^{-1})_* = (A_*)^{-1}$, thus, interchanging the roles of $X$ and $A_*X$ one has $\|X\|\leqslant c\|A_*X\|$.

Applying these observations to the tangent vectors in $\Psi_*(T^{(1,0)}\widetilde{\Sigma^{*}})$ and using the fact that $\pi_m$ is a Riemannian submersion (i.e. it preserves the lengths of horizontal vectors), one has that $g_{A\Phi}$ is quasi-isometric to $g_{\Phi}$, i.e. 
$$
\frac{1}{c}\leqslant\frac{\|d\Phi\|_g^2}{\|d(A\Phi)\|_g^2}\leqslant c.
$$
As a result, around every cone point the local expression for $\|d(A\Phi)\|_g^2$ is the same as for $\|d\Phi\|_g^2$, i.e. $A\Phi$ is a CTIHMM.

\end{proof}

In the present paper we investigate the simplest situation with nontrivial monodromy group. Consider the decomposition $\mathbb{C}^2\oplus\mathbb{C}^{2m-1}$. In the following we work with CTIHMM $\Phi\colon \Sigma^*\to\mathbb{S}^{2m}$ such that $M\subset \operatorname{SO}(2,\mathbb{R})$ acting on the first two components. We will act on such $\Phi$ by the group $\operatorname{SO}(2m-1,\mathbb{C})$ preserving the first two coordinates. It is easy to see that the action of $\operatorname{SO}(2m-1,\mathbb{C})$ preserves the monodromy representation. We would like to use an analogue of Theorem~\ref{S2deformation:thm} adapted to the present context in order to construct conical metrics with arbitrary many eigenfunctions. In particular, given a CTIHMM $\Phi\colon\Sigma^*\to\mathbb{S}^{2m}$ with monodromy described above, we would like to find $A_t\subset \operatorname{SO}(2m-1,\mathbb{C})$ such that $A_t\Phi\to \Phi_\infty$ as $t\to\infty$, where $\Phi_\infty$ is a map to $\mathbb{S}^2\subset\mathbb{S}^{2m}$. This way $g_{\Phi_\infty}$ is a reducible conical spherical metric with $2(m-1)$ extra eigenfunctions where the developing map is given similarly as~\eqref{e:develop}. Unfortunately, we could not prove the analogue of Theorem~\ref{S2deformation:thm} in full generality, so we prove it under an additional assumption.

In order to formulate the assumption, let us discuss the geometry of the multivalued twistor lift $\Psi\colon \Sigma^*\to \mH_m$. Note that the intersection $\Psi^*(p):=\Psi(p)\cap \mathbb{C}^{2m-1}$ is preserved under the action of the monodromy group, i.e. $\Psi^*(p)$ is a well defined isotropic subspace of $\mathbb{C}^{2m-1}$ for any $p\in\Sigma^*$.
Furthermore, $\dim_{\mathbb{C}}\Psi^*(p)$ is $m-2$, $m-1$ or $m$. Since $(2m-1)$-dimensional space can not contain $m$-dimensional isotropic subspace, the latter is impossible. 

\begin{theorem}
\label{CS2deformation:thm}
Let $\Phi\colon \Sigma^*\to \mathbb{S}^{2m}$, $m>1$ be a linearly full CTIHMM with non-trivial monodromy group $M\subset \operatorname{SO}(2,\mathbb{R})$. Assume that the map $\Psi^*$ continuously extends to points $\{p_j\}$ in the sense that for each $j$ there is an isotropic space $L_j'$ of dimension $(m-1)$ or $(m-2)$ such that $\lim_{p\to p_j}\Psi^*(p)\subset L'_j$. 
Then there exists a one-parameter subgroup $A_t\subset \operatorname{SO}(2m-1,\mathbb{C})$ such that $A_t\Phi\to \Phi_\infty$ in 
$C_{0}^{\infty}(\Sigma^{*})$-topology, where $\Phi_\infty\colon \Sigma^*\to \mathbb{S}^{2m}$ is a CTIHMM with its image in a proper totally geodesic subsphere $\mathbb{S}^{2}$. In particular, the metric $g_{\Phi_\infty}$ is a reducible conical spherical metric with the same monodromy representation as $\Phi$ and has $2(m-1)$ extra eigenfunctions.
\end{theorem} 
  


\begin{remark}
The harmonic map and the corresponding twistor lifts could a priori have very singular behaviour near the cone points. For example, the twistor lift could have an essential singularity there. In such a situation the dimension counting argument does not work. It seems likely that the condition $3)$ in the definition of CTIHMM prohibits such behaviour, but we could not verify this in general.  This is an analytic difficulty that could be resolved by finding good local representation for $\Phi$ or its twistor lift in a neighbourhood of a cone point. We remark that for $m=1$ the local representation in~\cite[Lemma 3.2]{CWWX} implies that the twistor lift (which in this case coincides with the developing map) can not have essential singularities at cone points. 
\end{remark}

\begin{remark}
The continuity assumption on $\Psi^*$ is satisfied for the examples we are considering below. At the same time, for these examples one can find the group $A_t$ explicitly, i.e. our construction does not require any general existence results.
\end{remark}

  
\begin{proof}
The proof is a variation on the dimension counting argument at the end of Section~\ref{TIHSM:sec}. The requirement that $A_t\subset \operatorname{SO}(2m-1,\mathbb{C})$ makes the dimension computations more elaborate. This is a purely geometric argument and it is presented below. 

Recall that $\Psi^*(p)$ is an isotropic subspace of $\mathbb{C}^{2m-1}$ of dimension $(m-2)$ or $(m-2)$. The set $\mA = \{L\in \mH_m|\,\dim(L\cap \mathbb{C}^{2m-1})=m-1\}\subset \mH_m$ is a proper submanifold, therefore, either $\dim \Psi^*(p) = m-1$ for all $p\in \Sigma^*$ or $\dim \Psi^*(p) = m-1$ only for a collection of isolated points in $\Sigma^*$.  
We claim that the former is impossible, and the only possibility is that $\dim \Psi^*(p) = m-1$ for a collection of isolated points $p\in \Sigma^*$ and  $\dim \Psi^*(p) = m-2$ otherwise.

Indeed, assume that $\dim \Psi^*(p) = m-1$ for all $p\in \Sigma^*$. Let $e_1, e_2$ be the basis of $\mathbb{C}^2$ and let $\Phi^*(p):=\pi_{m-1}\left(\Psi^*(p)\right)$ be the twistor projection of $\Psi^*(p)$ considered as an element of $\mH_{m-1}$. Then $\Psi(p) = \Psi^*(p) \oplus l(p)$, where $l(p)$ is an isotropic line in $\left(\Psi^*(p)\oplus\overline{\Psi^*(p)}\right)^\perp = \mathrm{span}_{\mathbb{C}}\{e_1,e_2,\Phi^*(p)\}$. Let $z$ be a local complex coordinate in the neighbourhood of $p$. Denote by $\partial_{\bar z}\Psi (p):=\{\partial_{\bar z}v(p),\,\,v(z)\in \Psi(z)\}$, where $v(z)$ is a family of vectors in $\mathbb{C}^{2m+1}$ and, similarly, $\partial_{z}\Psi (p):=\{\partial_{z}v(p),\,\,v(z)\in \Psi(z)\}$. Recall that $\Psi$ is a holomorphic map, which in local coordinates means that $\nabla_{\bar z}\Psi=0$ or, equivalently, $\partial_{\bar z}\Psi\subset \Psi$. Therefore, $\partial_{\bar z} \Psi^*\subset (\Psi\cap \mathbb{C}^{2m-1}) = \Psi^*(p)$, i.e. $\Psi^*$ is also holomorphic. Moreover, $\Psi$ is horizontal, which means that $\nabla_z\Psi$ lies in the horizontal distribution or, equivalently, $\partial_z\Psi\subset \Psi\oplus \mathbb{C}\Phi = (\bar\Psi)^\perp$. Therefore, 
$\partial_z\Psi^*\subset \mathbb{C}^{2m-1}\cap \bar\Psi^\perp = \mathbb{C}^{2m-1}\cap(\overline{\Psi^*})^\perp$, i.e. $\Psi^*$ is horizontal as a map to $\mH_{m-1}$. Since $\Psi^*$ is not constant (otherwise $\Psi$ is not linearly full), outside of a collection of isolated points one has $\partial_z\Psi^*= \Psi^*\oplus \mathbb{C}\Phi^*$, i.e. for such points $\mathbb{C}\Phi^*\in \partial_z\Psi\subset (\bar\Psi)^\perp$. Note that $\bar l\in \bar\Psi$, therefore, $\bar l\perp \Phi^*$. Since $\Phi^*$ is real, this implies $l\perp\Phi^*$ and, thus, $l\in \mathrm{span}_{\mathbb{C}}\{e_1,e_2\}$. As a result, $\Phi = \Phi^*$, which contradicts the fact that the monodromy group $M$ is non-trivial.

Thus, $\dim \Psi^*(p) = m-1$ for a collection of isolated points $p\in \Sigma^*$ and  $\dim \Psi^*(p) = m-2$ otherwise. The rest of the argument is almost identical to that at the end of Section~\ref{TIHSM:sec}. Let $L'\subset \mathbb{C}^{2m-1}$ be an $(m-1)$-isotropic plane and define $A_t$ to be $e^t \operatorname{id}$ on $L'$, $e^{-t}\operatorname{id}$ on $\bar L'$ and $\operatorname{id}$ on $(L'\oplus\bar L')^\perp$. It is sufficient to find $L'$ such that $\forall p\in \Sigma^*$ one has $\Psi^*(p)\cap L' =\{0\}$. The space of all isotropic $(m-1)$-planes in $\mathbb{C}^{2m-1}$ is isomorphic to $\mH_{m-1}$ and, therefore, has (complex) dimension $\frac{m(m-1)}{2}$. At the same time, $\mA = \{L'\in \mH_{m-1}|\,L'\cap \Psi^*(p)\ne \{0\},\,p\in \Sigma^*\}$ is a variety of (complex) dimension at most $\max\{1+(m-3), 0+ (m-2)\} + \frac{(m-2)(m-1)}{2} = \frac{m^2-m-2}{2}<\frac{m(m-1)}{2}$.

To complete the proof we need to control the behaviour of $\Psi^*$ near the cone points. This is where the additional assumption
comes in. Namely, assume that $\Psi^*$ can be continuously extended to points $\{p_j\}$ in the sense that for each $j$ there is an isotropic space $L_j'$ of dimension $(m-1)$ or $(m-2)$ such that $\lim_{p\to p_j}\Psi^*(p)\subset L'_j$. If $U_j\ni p_j$ denotes a sufficiently small neighbourhoods of $p_j$, then under this assumption $\mH_{m-1}\setminus \{L'\in \mH_{m-1}|\,L'\cap \Psi^*(p)\ne \{0\},\,p\in \cup (U_j\setminus\{p_j\})\}$ is still a $\frac{m(m-1)}{2}$-dimensional manifold. Thus, the dimension counting argument can be applied on the compact subset $\mathbb{S}^2\setminus (\cup\{U_j\})$ to complete the proof.
\end{proof}

In the next section we provide an explicit family of $\Phi$ satisfying the additional assumption for any value of $m$. As a result, one obtains examples of reducible conical spherical metrics with arbitrary number of extra eigenfunctions.


\subsection{Examples of CTIHMMs with non-trivial monodromy} 
\label{examples:sec}
We follow the approach of Barbosa in~\cite[Section 6]{Barbosa}. Let $\{e_0,\ldots,e_{2m}\}$ be an orthonormal basis of $\mathbb{R}^{2m+1}$. Set $E_j = \frac{1}{\sqrt{2}}(e_{2j-1}+ie_{2j})$, then the vectors $\{e_0, E_1,\bar E_1,\ldots E_m,\bar E_m \}$ form a basis of $\mathbb{C}^{2m+1}$, such that all  $(\cdot,\cdot)$ products vanish apart from $(e_0,e_0)=(E_1,\bar E_1)=\ldots=(E_m,\bar E_m)=1$. 

Identify $\mathbb{S}^2$ with the Riemann sphere $\mathbb{C}\cup\{\infty\}$ and set the non-integer cone points $p_1,p_2$ to be $0,\infty$. This way $\Sigma^* = \mathbb{C}\setminus\{0\}$ and we use the functions $z^\alpha$ on $\Sigma^*$ with non-integer $\alpha$ to construct our examples. Namely, consider a multivalued directrix curve $\Xi\colon \Sigma^*\to\mathbb{CP}^{2m}$ with the lift $\xi\colon\Sigma^*\to\mathbb{C}^{2m+1}$ of the form
\begin{equation*}
\begin{split}
\xi(z) &= a_0\bar E_1 + a_1\bar E_2z + \ldots + a_{m-2}\bar E_{m-1}z^{m-2} + a_{\alpha}\bar E_m z^{m-2+\alpha} + \\ 
&+e_0 z^k + E_m z^{2k-(m-2+\alpha)} +\ldots + E_2 z^{2k-1} + E_1 z^{2k},
\end{split}
\end{equation*}
where $0<\alpha<k-(m-2)$.
One can choose the coefficients $a_0,\ldots, a_{m-2}, a_\alpha$ so that $(\xi,\xi)=(\xi_z,\xi_z)=\ldots =(\xi^{(m-1)}_z,\xi^{(m-1)}_z)=0$. Indeed, this is an upper triangular system of linear equations with non-zero elements on the diagonal. Therefore it obviously has a unique solution. Let $\Psi$ be the corresponding map to $\mH_m$. Then, it is easy to see that the monodromy group of $\Psi$ is a rotation in the plane spanned by $\{E_m, \bar E_m\}$ and it is non-trivial as long as $\alpha\not\in\mathbb{Z}$.
Furthermore, since $\lim_{z\to 0}z^\beta = 0$ for $\beta>0$, the resulting twistor map $\Psi$ satisfies the additional assumption that $\Psi^*$ extends continuously to non-integer angle cone points. Thus, this map fits into the setup of the previous section and one can deform $\Psi$ using the one-parameter subgroup $A_t$. Let $L'$ be the $e^{-t}$-eigenspace of $A_t$, then $L'$ is an isotropic $(m-1)$-plane in $\mathrm{span}_{\mathbb{C}}\{E_m,\bar E_m\}^\perp$. Unfortunately, one can not take $L' = L'_0:=\mathrm{span}\{E_1,\ldots, E_{m-1}\}$. Indeed, $\Psi^*(\infty) = L'_0$, which is not allowed by the argument in Section~\ref{CTIHMM:sec}. Similarly, if $a_j\ne 0$ for $j=0,\ldots, m-2, \alpha$, then $\Psi^*(0) = \bar L'_0$. Now we give the following example for $m=2$:

\begin{example}
\label{ex:m=2}
Let $m=2$. Then the system $(\xi,\xi) = (\xi',\xi')=0$ is reduced to 
\begin{equation*}
\begin{cases}
&2a_0 + 2a_\alpha + 1 = 0,\\
&	    2\alpha(2k-\alpha)a_\alpha + k^2 = 0.
\end{cases}
\end{equation*}
Solving the system yields
\begin{equation*}
\begin{cases}
&\xi(z) = \frac{(k-\alpha)^2}{2\alpha(2k-\alpha)}\bar E_1 - \frac{k^2}{2\alpha(2k-\alpha)}z^\alpha\bar E_2 + z^ke_0 + z^{2k-\alpha} E_2 + z^{2k} E_1;\\
&\xi'(z) = -\frac{k^2}{2(2k-\alpha)} z^{\alpha-1}\bar E_2 + kz^{k-1}e_0 + (2k-\alpha)z^{2k-\alpha-1} E_2 + 2kz^{2k-1} E_1.
\end{cases}
\end{equation*}
After a simple change of basis in $\Psi=\mathrm{span}_{\mathbb{C}}\{\xi,\xi'\}$ one has that $\Psi(z)$ is spanned by
\begin{equation*}
\begin{cases}
& -\frac{k^2}{2(2k-\alpha)} z^{\alpha-1}\bar E_2+ (2k-\alpha)z^{2k-\alpha-1} E_2 + \ldots;\\
& 0\cdot\bar E_2 + \left(1 - \frac{2k-\alpha}{\alpha}\right) E_2 +\ldots,
\end{cases}
\end{equation*}
where the components of the vectors in $\mathrm{span}_{\mathbb{C}}\{E_1,\bar E_1, e_0\}$ are not written explicitly. Since $\alpha\ne k$, this basis shows that for $z\ne\{0,\infty\}$ one has $\Psi^*(z) =\{0\}$. At the same time, it was noted above that $\Psi^*(0) = \mathbb{C}\bar E_1$, $\Psi^*(0) = \mathbb{C}E_1$. Therefore, one can choose $L'$ to be any isotropic line in $\mathrm{span}_{\mathbb{C}}\{E_1,\bar E_1, e_0\}$ not spanned by $E_1$ or $\bar E_1$.

Let us choose $E_1' = \frac{1}{\sqrt{2}}e_0 + \frac{1}{2}(E_1'-\bar E_1)$ and set $L' = \mathbb{C}E_1'$. It is convenient to do the computations in the basis $\{e'_0,E_1', \bar E_1'\}$, where
\begin{equation*}
\begin{cases}
&\bar E_1 = \frac{1}{\sqrt{2}}e'_0 -\frac{1}{2}(E_1' - \bar E_1'),\\
& E_1 = \frac{1}{\sqrt{2}}e'_0 + \frac{1}{2}(E_1' - \bar E_1'),\\
&e_0 = \frac{1}{\sqrt{2}}(E_1' + \bar E_1'),
\end{cases}
\end{equation*}
so that $(E_1',\bar E_1') = (e_0,e_0) = 1$ and all other pairwise products vanish. In the new basis one has
\begin{equation*}
\begin{split}
\xi(z)= &\left(\frac{(k-\alpha)^2}{4\alpha(k-\alpha)} + \frac{1}{\sqrt{2}}z^k - \frac{1}{2}z^{2k} \right)\bar E_1' - \frac{k^2}{2\alpha(2k-\alpha)}z^\alpha\bar E_2\\
&+ \frac{1}{\sqrt{2}}\left( \frac{(k-\alpha)^2}{2\alpha(k-\alpha)} +z^{2k}  \right)e'_0 + z^{2k-\alpha}E_2\\
&  + \left(-\frac{(k-\alpha)^2}{4\alpha(k-\alpha)} + \frac{1}{\sqrt{2}}z^k + \frac{1}{2}z^{2k}\right)E'_1;\\
\xi'(z) =& \left( \frac{k}{\sqrt{2}}z^{k-1} - kz^{2k-1}\right)\bar E'_1 -\frac{k^2}{2(2k-\alpha)} z^{\alpha-1}\bar E_2 + \sqrt{2}kz^{2k-1}e_0' \\
&+(2k-\alpha)z^{2k-\alpha-1} E_2 + \left( \frac{k}{\sqrt{2}}z^{k-1} + kz^{2k-1}\right) E'_1.
\end{split}
\end{equation*}
Finally, one can find a basis in $\Psi(z)$ of the form~\eqref{specialbasis:eq} and use the formula~\eqref{e:develop} and~\eqref{e:eigenfunctions} to find the limiting developing map and the extra eigenfunctions. After a lengthy but elementary computation one obtains the developing map
$$
f(z) = cz^{k-\alpha}\frac{z^k -\frac{\alpha-k}{\sqrt{2}\alpha}}{z^k - \frac{k-\alpha}{\sqrt{2}(2k-\alpha)}},
$$
where $c$ is a non-zero constant. After setting $\beta = k-\alpha$ and doing the change of variables $z\mapsto\lambda z$, where $\lambda^k = \frac{k-\alpha}{\sqrt{2}(2k-\alpha)}$ one recovers the developing map of Example~\ref{intro:example}. Further computations also yield the extra eigenfunctions presented there.


\end{example}

\subsection{Algorithm for constructing metrics with arbitrarily many extra eigenfunctions}
\label{s:algorithm}
The computations similar to those in Example~\ref{ex:m=2} can be performed for $m>2$. However, the computations quickly become too lengthy to manage by hand. Fortunately, most of the steps are very explicit and can be automated on the computer. The only part of the algorithm that is not explicit is the choice of $L'$. Computing $\Psi^*$ involves solving a system of non-linear equation and, thus, could be computationally demanding. At the same time, the proof of Theorem~\ref{S2deformation:thm} implies that a "generic" choice of $L'$ should work. As a result,  the following algorithm seems to be best suited in practice,
\begin{enumerate}
\item Choose $m,k,\alpha$ and solve the triangular linear system to find $a_i$;
\item Choose a fixed $L'$ such that $L'\cap (L_0'\cup \bar L_0')=\{0\}$;
\item Find the special basis of the form~\eqref{specialbasis:eq};
\item Check that the formula~\eqref{e:eigenfunctions} gives extra eigenfunctions, i.e. that they are bounded on $\overline{\mathbb{C}}$ and in the kernel of $\Delta-2$;
\item If Step $(4)$ does not give an extra eigenfunction, go back to Step $(2)$ and pick a different $L'$.
\end{enumerate}

%

\bibliography{ConicalMetrics}

\begin{thebibliography}{CWWX15}

\bibitem[Bar75]{Barbosa}
Jo\~{a}o Lucas~Marqu\^{e}s Barbosa.
\newblock On minimal immersions of {$S^{2}$} into {$S^{2m}$}.
\newblock {\em Trans. Amer. Math. Soc.}, 210:75--106, 1975.

\bibitem[BDMM11]{BMM}
Daniele Bartolucci, Francesca De~Marchis, and Andrea Malchiodi.
\newblock Supercritical conformal metrics on surfaces with conical
  singularities.
\newblock {\em Int. Math. Res. Not. IMRN}, (24):5625--5643, 2011.

\bibitem[Cal67]{Calabi}
Eugenio Calabi.
\newblock Minimal immersions of surfaces in {E}uclidean spheres.
\newblock {\em J. Differential Geometry}, 1:111--125, 1967.

\bibitem[CWWX15]{CWWX}
Qing Chen, Wei Wang, Yingyi Wu, and Bin Xu.
\newblock Conformal metrics with constant curvature one and finitely many
  conical singularities on compact {R}iemann surfaces.
\newblock {\em Pacific J. Math.}, 273(1):75--100, 2015.

\bibitem[Eji94]{Ejiri0}
Norio Ejiri.
\newblock Minimal deformation of a nonfull minimal surface in {$S^4(1)$}.
\newblock {\em Compositio Math.}, 90(2):183--209, 1994.

\bibitem[Eji98]{Ejiri}
Norio Ejiri.
\newblock The boundary of the space of full harmonic maps of {$S^2$} into
  {$S^{2m}(1)$} and extra eigenfunctions.
\newblock {\em Japan. J. Math. (N.S.)}, 24(1):83--121, 1998.

\bibitem[Ere04]{Eremenko}
Alexandre Eremenko.
\newblock Metrics of positive curvature with conic singularities on the sphere.
\newblock {\em Proc. Amer. Math. Soc.}, 132(11):3349--3355, 2004.

\bibitem[Ere21]{Ere2}
Alexandre Eremenko.
\newblock Metrics of constant positive curvature with conic singularities. a
  survey.
\newblock {\em arXiv preprint arXiv:2103.13364}, 2021.

\bibitem[GO93]{GuestOhnita}
Martin~A. Guest and Yoshihiro Ohnita.
\newblock Group actions and deformations for harmonic maps.
\newblock {\em J. Math. Soc. Japan}, 45(4):671--704, 1993.

\bibitem[Kar21]{KarpukhinRP2}
Mikhail Karpukhin.
\newblock Index of minimal spheres and isoperimetric eigenvalue inequalities.
\newblock {\em Invent. Math.}, 223(1):335--377, 2021.

\bibitem[KNPP17]{KNPP}
Mikhail Karpukhin, Nikolai Nadirashvili, Alexei~V Penskoi, and Iosif
  Polterovich.
\newblock An isoperimetric inequality for laplace eigenvalues on the sphere.
\newblock {\em To appear in J. Differential Geom.}, arXiv:1706.05713, 2017.

\bibitem[Kot97]{Kotani}
Motoko Kotani.
\newblock Harmonic {$2$}-spheres with {$r$} pairs of extra eigenfunctions.
\newblock {\em Proc. Amer. Math. Soc.}, 125(7):2083--2092, 1997.

\bibitem[MP16]{MP}
Gabriele Mondello and Dmitri Panov.
\newblock Spherical metrics with conical singularities on a 2-sphere: angle
  constraints.
\newblock {\em Int. Math. Res. Not. IMRN}, (16):4937--4995, 2016.

\bibitem[MP19]{MP2}
Gabriele Mondello and Dmitri Panov.
\newblock Spherical surfaces with conical points: systole inequality and moduli
  spaces with many connected components.
\newblock {\em Geom. Funct. Anal.}, 29(4):1110--1193, 2019.

\bibitem[MR91]{MR}
Sebasti\'{a}n Montiel and Antonio Ros.
\newblock Schr\"{o}dinger operators associated to a holomorphic map.
\newblock In {\em Global differential geometry and global analysis ({B}erlin,
  1990)}, volume 1481 of {\em Lecture Notes in Math.}, pages 147--174.
  Springer, Berlin, 1991.

\bibitem[MZ19]{MZ2}
Rafe Mazzeo and Xuwen Zhu.
\newblock Conical metrics on {R}iemann surfaces, {II}: spherical metrics.
\newblock {\em To appear in Int. Math. Res. Not.}, arXiv:1906.09720, 2019.

\bibitem[MZ20]{MZ1}
Rafe Mazzeo and Xuwen Zhu.
\newblock Conical metrics on {R}iemann surfaces {I}: {T}he compactified
  configuration space and regularity.
\newblock {\em Geom. Topol.}, 24(1):309--372, 2020.

\bibitem[Nay93]{Nayatani}
Shin Nayatani.
\newblock Morse index and {G}auss maps of complete minimal surfaces in
  {E}uclidean {$3$}-space.
\newblock {\em Comment. Math. Helv.}, 68(4):511--537, 1993.

\bibitem[NS19]{NS}
Shin Nayatani and Toshihiro Shoda.
\newblock Metrics on a closed surface of genus two which maximize the first
  eigenvalue of the {L}aplacian.
\newblock {\em C. R. Math. Acad. Sci. Paris}, 357(1):84--98, 2019.

\bibitem[Tro89]{Troyanov}
Marc Troyanov.
\newblock Metrics of constant curvature on a sphere with two conical
  singularities.
\newblock In {\em Differential geometry ({P}e\~{n}\'{\i}scola, 1988)}, volume
  1410 of {\em Lecture Notes in Math.}, pages 296--306. Springer, Berlin, 1989.

\bibitem[UY00]{UY}
Masaaki Umehara and Kotaro Yamada.
\newblock Metrics of constant curvature {$1$} with three conical singularities
  on the {$2$}-sphere.
\newblock {\em Illinois J. Math.}, 44(1):72--94, 2000.

\bibitem[XZ19]{XZ}
Bin Xu and Xuwen Zhu.
\newblock Spectral properties of reducible conical metrics.
\newblock {\em To appear in Illinois J. Math.}, arXiv:1909.00546, 2019.

\end{thebibliography}
\bibliographystyle{alpha}
\end{document}